\documentclass[letterpaper,11pt]{amsart}

\usepackage{graphicx,color,epsfig}
\usepackage{amsfonts,amsmath,amssymb,amsthm}

\newtheorem{thm}{Theorem}[section]

\newtheorem{lem}[thm]{Lemma}
\newtheorem{cor}[thm]{Corollary}

\newtheorem{asm}{Assumption}

\theoremstyle{remark}
\newtheorem{rem}[thm]{Remark}

\newtheorem{ex}{Example}[section]

\theoremstyle{definition}

\newcommand{\ra}{\rightarrow}

\newcommand{\N}{\mathbb N}     
\newcommand{\R}{\mathbb R}     
\newcommand{\Z}{\mathbb Z}     

\renewcommand{\a}{\alpha}

\renewcommand{\d}{\delta}

\newcommand{\e}{\varepsilon}

\renewcommand{\l}{\lambda}
\renewcommand{\L}{\Lambda}

\newcommand{\s}{\sigma}

\renewcommand{\k}{\kappa}

\newcommand{\bigo}{\mathcal{O}}

\newcommand{\fl}[1]{\lfloor #1 \rfloor}  
\newcommand{\ind}[1]{ \mathbf{1}_{ \{ #1 \} } } 

\newcommand{\be}{\begin{equation}}
\newcommand{\ee}{\end{equation}}

\newcommand{\w}{\omega}              
\newcommand{\E}{\mathbb{E}}          
\newcommand{\Pv}{\mathbf{P}}		
\newcommand{\Ev}{\mathbf{E}}

\begin{document}

\title[Large deviations and slowdowns]{Large deviations and slowdown asymptotics for one-dimensional excited random walks}
\author{Jonathon Peterson}
\address{Jonathon Peterson \\  Purdue University \\ Department of Mathematics \\ 150 N University Street \\ West Lafayette, IN  47907 \\ USA}
\email{peterson@math.purdue.edu}
\urladdr{http://www.math.purdue.edu/~peterson}
\thanks{J. Peterson was partially supported by National Science Foundation grant DMS-0802942.}


\date{\today}

\begin{abstract}
We study the large deviations of excited random walks on $\mathbb{Z}$. We prove a large deviation principle for both the hitting times and the position of the random walk and give a qualitative description of the respective rate functions. When the excited random walk is transient with positive speed $v_0$, then the large deviation rate function for the position of the excited random walk is zero on the interval $[0,v_0]$ and so probabilities such as $P(X_n < nv)$ for $v \in (0,v_0)$ decay subexponentially. We show that rate of decay for such slowdown probabilities is polynomial of the order $n^{1-\delta/2}$, where $\delta>2$ is the expected total drift per site of the cookie environment.  
\end{abstract}

\maketitle

\section{Introduction}

In this paper we study the large deviations for one-dimensional excited random walks. Excited random walks are a model for a self-interacting random walk, where the transition probabilities depend on the number of prior visits of the random walk to the current site. 
The most general model for excited random walks on $\Z$ is the following. 
Let $\Omega = [0,1]^{\Z \times \N}$, and for any element $\w = \{\w_{i}(j) \}_{i \in \Z,\, j \geq 1} \in \Omega$ we can define an excited random walk $X_n$ by letting $\w_i(j)$ be the probability that the random walk moves to the right upon its $j$-th visit to the site $i\in \Z$. More formally, we will let $P_\w(X_0 = 0)$ and 
\[
 P_\w(X_{n+1} = X_n+1 | \, \mathcal{F}_n ) 
= 1- P_\w(X_{n+1} = X_n-1 | \, \mathcal{F}_n )
= \w_x\left( \# \{k \leq n: \, X_k = x \} \right),
\]
where $\mathcal{F}_n = \s(X_0,X_1,\ldots,X_n)$. 
Note that the excited random walk $X_n$ is not a Markov chain since the transition probabilities depend on the entire past of the random walk and not just the current location. 

Excited random walks are also sometimes called \emph{cookie random walks}, since one imagines a stack of ``cookies'' at every site which each induce a specific bias to the walker. 
When the walker visits the site $x$ for the $i$-th time, he eats the $i$-th cookie which causes his next step to be as a simple random walk with parameter $\w_x(i)$. For this reason we will also refer to $\w = \{\w_i(j)\}_{i\in\Z, \, j \geq 1}$ as a \emph{cookie environment}. 

We can also assume that the cookie environment $\w$ is first chosen randomly. That is, let $\Pv$ be a probability distribution on the space of cookie environments $\Omega$, and define a new measure on the space of random walk paths $\Z^{\Z_+}$ by averaging over all cookie environments. That is, let
\[
 P(\cdot) = \int_\Omega P_\w(\cdot) \, \Pv(d\w). 
\]
For a fixed cookie environment $\w$, the law $P_\w$ is referred to as the \emph{quenched} law of the excited random walk, and $P$ is called the \emph{averaged} law of the excited random walk.

Most of the results for excited random walks make the assumption that there are only finitely many cookies per site. That is, there exists an $M$ such that $\w_i(j) = 1/2$ for any $i\in\Z$ and $j > M$ so that after $M$ visits to any site the transitions are like a simple symmetric random walk. 
\begin{asm}\label{Mcookieasm}
There exists an integer $M<\infty$ such that there are almost surely only $M$ cookies per site. That is, $\Pv(\Omega_M) = 1$, where 
\[
 \Omega_M = \Omega \cap \{ \w: \, \w_i(j) = 1/2, \, \forall i\in\Z,\, \forall j>M \}. 
\]
\end{asm}
\noindent
We will also make the common assumption that the cookie environment is i.i.d.\ in the following sense. 
\begin{asm}\label{iidasm}
 The distribution $\Pv$ is such that the sequence of cookie environments at each site $\{ \w_i(\cdot) \}_{i\in\Z}$ is i.i.d.
\end{asm}
Finally, we will make the following non-degeneracy assumption on cookie environments. 
\begin{asm}\label{easm}
 With $M$ as in Assumption \ref{Mcookieasm}, 
\[
 E\left[ \prod_{j=1}^M \w_0(j) \right] > 0 \quad \text{and} \quad E\left[ \prod_{j=1}^M (1-\w_0(j)) \right] > 0.
\]
\end{asm}

Excited random walks were first studied by Benjamini and Wilson in \cite{bwERW}, where they considered the case of deterministic cookie environments with one cookie per site (that is $M=1$). The focus of Benjamini and Wilson was mainly on the $\Z^d$ case, but in the special case of $d=1$ they showed that excited random walks with one cookie per site are always recurrent. The model was further generalized by Zerner in \cite{zMERW} to allow for multiple cookies per site and for randomness in the cookie environment, but with the restriction that all cookies induced a non-negative drift (that is $\w_i(j) \geq 1/2$). Recently the model of excited random walks was further generalized by Zerner and Kosygina to allow for cookies with both positive and negative drifts \cite{kzPNERW}. 

The recurrence/transience and limiting speed for one-dimensional excited random walks have been studied in depth under the above assumptions. 
A critical parameter for describing the behavior of the excited random walk is the expected total drift per site
\be\label{ddef}
 \d = \Ev\left[ \sum_{i\geq 1} ( 2\w_0(j) - 1) \right] = \Ev\left[ \sum_{i=1}^M ( 2\w_0(j) - 1) \right].
\ee
Zerner showed in \cite{zMERW} that excited random walks with all cookies $\w_i(j)\geq 1/2$ are transient to $+\infty$ if and only if $\d>1$. 
Additionally, Zerner showed that the limiting speed $v_0 = \lim_{n\ra\infty} X_n / n$ exists, $P$-a.s., but was not able to determine when the speed is non-zero. Basdevant and Singh solved this problem in \cite{bsCRWspeed} where they showed that $v_0>0$ if and only if $\d>2$. These results for recurrence/transience and the limiting speed were given only for cookies with non-negative drift but were recently generalized by Kosygina and Zerner \cite{kzPNERW} to the general model we described above that allows for cookies with both positive and negative drifts. In summary, under Assumptions \ref{Mcookieasm} -- \ref{easm}, the following results are known. 
\begin{itemize}
 \item $X_n$ is recurrent if and only if $\d\in[-1,1]$. Moreover, 
\[
 \lim_{n\ra\infty} X_n = 
\begin{cases}
 -\infty & \text{ if } \d< -1 \\
 +\infty & \text{ if } \d> 1, 
\end{cases}
\qquad P\text{-a.s.}
\]
 \item There exists a constant $v_0$ such that $\lim_{n\ra\infty} X_n/n = v_0$, $P$-a.s. Moreover, $v_0 = 0$ if and only if $\d \in [-2,2]$. 
\end{itemize}
Limiting distributions for excited random walks are also known with the type of rescaling and limiting distribution depending only on the parameter $\d$ given in \eqref{ddef}. The interested reader is referred to the papers \cite{bsRGCRW,kzPNERW,kmLLCRW,dCLTERW,dkSLRERW} for more information on limiting distributions.

\subsection{Main Results}
In this paper, we are primarily concerned with the large deviations of excited random walks. In a similar manner to the approach used for large deviations  of random walks in random environments, we deduce a large deviation principle for $X_n/n$ from a large deviation principle for $T_n/n$, where 
\[
 T_n = \inf\{ k \geq 0: \, X_k = n \}, \quad n \in \Z
\]
are the hitting times of the excited random walk. 
However, we do not prove a large deviation principle for the hitting times directly.
Instead, we use an associated branching process with migration $V_i$ that has been used previously in some of the above mentioned papers on the speed and limiting distributions for excited random walks \cite{bsCRWspeed, kzPNERW, kmLLCRW}. We prove a large deviation principle for $n^{-1}\sum_{i=1}^n V_i$ and use this to deduce a large deviation principle for $T_n/n$ which in turn implies the following large deviation principle for $X_n/n$. 

\begin{thm}\label{LDPXn}
The empirical speed of the excited random walk $X_n/n$ satisfies a large deviation principle with rate function $I_X$ defined in \eqref{IXdef}. 
That is, for any open set $G \subset [-1,1]$, 
\be\label{LDPXnlb}
 \liminf_{n\ra\infty} \frac{1}{n} \log P( X_n/n \in G) \geq - \inf_{x \in G} I_X(x), 
\ee
and for any closed set $F \subset [-1,1]$, 
\[
 \limsup_{n\ra\infty} \frac{1}{n} \log P( X_n/n \in F) \leq - \inf_{x \in F} I_X(x). 
\]
\end{thm}
\begin{rem}
After the initial draft of this paper was completed, it was noted that a general large deviation principle for certain non-Markovian random walks due to Rassoul-Agha \cite{rLDPNMRW} can be used to prove Theorem \ref{LDPXn} in certain cases. 
Thus, it is necessary to point out some of the differences with the current paper. 
\begin{itemize}
 \item In \cite{rLDPNMRW} the random walks are assumed to be \emph{uniformly elliptic}, which in the context of this paper would require $\w_i(j) \in [c,1-c]$ for all $i\in\Z$, $j\geq 1$ and some $c>0$. In contrast, we only assume the weaker condition in Assumption \ref{easm}. 
 \item The results of \cite{rLDPNMRW} only apply directly to excited random walks with deterministic cookie environments.
If the cookie environments are allowed to be random and satisfying Assumption \ref{iidasm}, then a technical difficulty arises in satisfying one of the conditions 
for the large deviation principle in \cite{rLDPNMRW}. 
Specifically, the transition probabilities $q(\mathrm{w},z)$ for the shifted paths as defined in \cite{rLDPNMRW} do not appear to be continuous in $\mathrm{w}$ for the required topology. 
We suspect, however, that the techniques of \cite{rLDPNMRW} could be adapted to apply to this case as well. 
 \item The formulation of the large deviation rate function in \cite{rLDPNMRW} is difficult to work with and the only stated properties of the rate function are convexity and a description of the zero set. In contrast, our method gives a more detailed description of the rate function (see Lemma \ref{IXprop} and Figure \ref{IXfig}). 
 \item The method in \cite{rLDPNMRW} does not also give a large deviation principle for the hitting times of the random walk, though one could use an argument similar to that in Section \ref{LDPXsec} below to deduce a large deviation principle for the hitting times from the large deviation principle for the location of the random walk. 
\end{itemize}
\end{rem}

As mentioned in the above remark, the formulation of the rate function $I_X$ given in the proof of Theorem \ref{LDPXn} allows us to give a good qualitative description of the rate function (see Lemma \ref{IXprop}). 
One particularly interesting property is that if $\d>2$ (so that the limiting speed $v_0>0$) then $I_X(x) = 0$ when $x \in [0,v_0]$. Thus, probabilities of the form $P(X_n < nx)$ decay subexponentially if $x \in (0,v_0)$. In fact, as the following example shows, one can see quite easily that such slowdown probabilities must have a subexponential rate of decay. 
\begin{ex}\label{slowdownex}
We exhibit a naive strategy for obtaining a slowdown of the excited random walk. 
Consider the event where the excited random walk first follows a deterministic path that visits every site in $[0,n^{1/3})$ $M$ times (so that no cookies remain in the interval) and then the random walk stays in the interval $[0,n^{1/3})$ for $n$ steps.
The probabilistic cost of forcing the random walk to follow the deterministic path at the beginning is $e^{-c'Mn^{1/3}}$ for some $c'>0$. Then, since there are no cookies left in the interval, the probability of then staying in $[0,n^{1/3})$ for $n$ steps before exiting to the right is a small deviation computation for a simple symmetric random walk. The probability of this event can be bounded below by $C e^{-c'' n^{1/3}}$ for some $C,c''>0$ (see Theorem 3 in \cite{mSmDev}). Thus, the total probability of the above event for the excited random walk is at least $C e^{-c n^{1/3}}$. 
\end{ex}
The example above shows that $P( X_n < xn )$ decays slower than a stretched exponential. However, this strategy turns out to be far from the optimal way for obtaining such a slowdown.
The second main result of this paper is that the true rate of decay for slowdowns is instead polynomial of the order $n^{1-\d/2}$.
\begin{thm}\label{Slowdownthm}
 If $\d>2$, then
\be\label{Xnslowdown}
 \lim_{n\ra\infty} \frac{ \log P(X_n < nx) } {\log n} = 1 - \frac{\d}{2}, \quad \forall x \in (0,v_0), 
\ee
and
\be\label{Tnslowdown}
 \lim_{n\ra\infty} \frac{ \log P( T_n > nt ) }{ \log n } = 1 -\frac{\d}{2}, \quad \forall t > 1/v_0. 
\ee
\end{thm}

\subsection{Comparison with RWRE}
Many of the prior results for one-dimensional excited random walks are very similar to the corresponding statements for random walks in random environments (RWRE). For instance, both models can exhibit transience with sublinear speed and they have the same types limiting distributions for the  hitting times and the location of the random walk \cite{kzPNERW,sRWRE,kksStable}. 
Thus, it is interesting to compare the results of this paper with what is known for one-dimensional RWRE. 

Large deviations for one-dimensional RWRE (including a qualitative description of the rate functions) were studied in \cite{cgzLDP} and subexponential slowdown asymptotics for ballistic RWRE similar to Theorem \ref{Slowdownthm} were studied in \cite{dpzASlowdown}. The similarities to the current paper are greatest when the excited random walk has $\d>2$ and the RWRE is transient with positive speed and ``nestling'' (i.e., the environment has positive and negative drifts). In this case, the large deviation rate function for either model is zero on the interval $[0,v_0]$, where $v_0 = \lim_{n\ra\infty} X_n/n$ is the limiting speed. 
Moreover, the polynomial rates of decay of the slowdown probabilities are related to the limiting distributions of the random walks in the same way. For instance, in either model if the slowdown probabilities decay like $n^{1-\a}$ with $\a \in (1,2)$ then $n^{-1/\a}(X_n -n v_0)$ converges in distribution to an $\a$-stable random variable \cite{kzPNERW,kksStable}.

An interesting difference in the rate functions for excited random walks and RWRE is that $I'_X(0) = 0$ in the present paper, while for transient RWRE the left and right derivatives of the rate function are not equal at the origin \cite{cgzLDP}.
Since (in both models) $I_X$ is defined in terms of the large deviation rate function $I_T(t)$ for the hitting times $T_n/n$, this is related to the fact that
$\inf_t I_T(t) = 0$ for excited random walks (see Lemma \ref{LDPTn}) while the corresponding rate function for the hitting times of RWRE is uniformly bounded away from $0$ if the walk is transient to the left.

\subsection{Outline}
The structure of the paper is as follows. 
In Section \ref{Visec} we define the associated branching process with migration $V_i$, mention its relationship to the hitting times of the excited random walk, and prove a few basic properties about the process $V_i$. Then in Section \ref{LDPVsec} we prove a large deviation principle for the empirical mean of the process $V_i$ and prove some properties of the corresponding rate function. 
The large deviation principle for the empirical mean of the process $V_i$ is then used to deduce large deviation principles for $T_n/n$ and $X_n/n$ in Sections \ref{LDPTsec} and \ref{LDPXsec}, respectively.
Finally, in Section \ref{Slowdownsec} we prove the subexponential rate of decay for slowdown probabilities.

\section{A related branching process with random migration}\label{Visec}

In this section we recall how the hitting times $T_n$ of the excited random walk can be related to a branching process with migration. 
%
We will construct the related branching process with migration using the ``coin tossing'' construction that was given in \cite{kzPNERW}. 
Let a cookie environment $\w = \{\w_i(j)\}_{i\in\Z, j\geq 1}$ be fixed, and let $\{\xi_{i,j} \}_{i\in\Z, j\geq 1}$ be an independent family of Bernoulli random variables with $P(\xi_{i,j} = 1) = \w_i(j)$. For $i$ fixed, we say that the $j$-th Bernoulli trial is a ``success'' if $\xi_{i,j} = 1$ and a ``failure'' otherwise. Then, let $F^{(i)}_m$ be the number of failures in the sequence $\{\xi_{i,j}\}_{j\geq 1}$ before the $m$-th success. That is, 
\[
 F^{(i)}_m = \min \left\{ \ell \geq 1: \sum_{j=1}^\ell \xi_{i,j} = m \right\} - m.
\]
Finally, we define the branching process with migration $\{V_i \}_{i\geq 1}$ by
\[
 V_0 = 0, \quad\text{and}\quad V_{i+1} = F^{(i)}_{V_i + 1}, \, \text{ for } i \geq 0. 
\]
If the $\w_i(j)$ were all equal to $1/2$ then the process $\{V_i\}$ would be a critical Galton-Watson branching process with one additional immigrant per generation. Allowing the first $M$ cookie strengths at each site to be different than $1/2$ has the effect of making the migration effect more complicated (in particular, the migration in each generation is random and can depend on the current population size). We refer the interested reader to \cite{bsCRWspeed} for a more detailed description of the interpretation of $V_i$ as a branching process with migration. 

In addition to the above branching process with migration, we will also need another branching process with a random initial population and one less migrant each generation. For any $n\geq 1$, let $V^{(n)}_0 = V_n$ where $V_n$ is constructed as above and let $V^{(n)}_i = F^{(n+i-1)}_{V^{(n)}_{i-1}}$, where we let $F^{(i)}_0 = 0$. 
Note that with this construction, we have that $V^{(n)}_i \leq V_{n+i}$ for all $i$. Moreover, while the Markov chain $V_i$ is irreducible, the lack of the extra migrant each generation makes $0$ an absorbing state for $V^{(n)}_i$. 

The relevance of the processes $\{V_i\}_{i\geq 0}$ and $\{V_i^{(n)}\}_{i\geq 0}$ to the hitting times $T_n$ of the excited random walk is the following. 
\be\label{TnVirep}
 T_n  \overset{\mathcal{D}}{=} n + 2 \sum_{i=1}^n V_i + 2 \sum_{i=1}^\infty V_i^{(n)}. 
\ee
To explain this relation let
$U^n_i = \# \{k\leq T_n \, : \, X_k = i, \, X_{k+1} = i-1 \}$ be the number of times the random walk jumps from $i$ to $i-1$ before time $T_n$. 
Then, it is easy to see that $T_n = n + 2 \sum_{i\leq n} U^n_i$ and \eqref{TnVirep} follows from the fact that
\be\label{UiVirep}
 (U^n_n, U^n_{n-1}, \ldots U^n_1, U^n_0, U^n_{-1}, U^n_{-2},\ldots ) \overset{\mathcal{D}}{=} (V_1, V_2, \ldots, V_{n-1}, V_n, V^{(n)}_1, V^{(n)}_2,\ldots).  
\ee
The details of the above joint equality in distribution can be found in \cite{bsCRWspeed} or \cite{kmLLCRW}. 
\begin{rem}
 Technically, the relation \eqref{TnVirep} is proved in \cite{bsCRWspeed} and \cite{kmLLCRW} only in the cases where $T_m < \infty$ with probability one. 
However, an examination of the proof shows that $P(T_n = k) = P(n + \sum_{i=1}^n V_i + 2 \sum_{i=1}^\infty V_i^{(n)} = k)$ for any finite $k$ and so both sides of \eqref{TnVirep} are infinite with the same probability as well. 
\end{rem}

\subsection{Regeneration structure}

We now define a sequence of regeneration times for the branching process $V_i$. Let $\s_0 = 0$ and for $k\geq 1$
\[
 \s_k = \inf \{i>\s_{k-1}: V_i = 0\}. 
\]
Also, for $k\geq 1$ let 
\[
 W_k = \sum_{i=1}^{\s_{k}} V_i
\]
be the total offspring of the branching process by the $k^{th}$ regeneration time.
The tails of $\s_1$ and $W_1$ were analyzed in \cite{kmLLCRW} in the case when $\d>0$. 
\begin{lem}[Theorems 2.1 and 2.2 in \cite{kmLLCRW}]
 If $\d>0$ then, 
\be\label{Wstails}
 P(\s_1 > x) \sim C_1 x^{-\d} \quad\text{and}\quad P(W_1 > x ) \sim C_2 x^{-\d/2} \quad \text{as } x\ra\infty. 
\ee
\end{lem}

Note that if the Markov chain $V_i$ is transient, then eventually $\s_k = W_k = \infty$ for all $k$ large enough. The following Lemma specifies the recurrence/transience properties of the Markov chain $V_i$. 
\begin{lem}\label{Vrt}
 The Markov chain $V_i$ is recurrent if and only if $\d\geq 0$ and positive recurrent if and only if $\d > 1$. 
\end{lem}
\begin{proof}
The tail decay of $\s_1$ shows that $E[\s_1] < \infty$ if $\d> 1$ and $E[\s_1] = \infty$ if $\d \in (0,1]$. Therefore, it is enough to show that $V_i$ is recurrent if and only if $\d \geq 0$. 
This can be proven by an appeal to some previous results on branching proceses with migration as was done in \cite{kzPNERW}. 
A small difficulty arises in that the distribution of the migration that occurs before the generation of the $(i+1)$-st generation depends on the population of $i$-th generation. 
However, this can be dealt with in the same manner as was done in \cite{kzPNERW}. To see this, let 
$\widehat{V}_i$ be defined by 
\[
 \widehat{V}_0 = 0, \quad \text{and}\quad \widehat{V}_{i+1} = F^{(i)}_{(\widehat{V}_i + 1) \vee M}.
\]
Note that $V_i$ and $\widehat{V}_i$ have the same transition probabilities when starting from a site $k\geq M-1$, and thus $V_i$ and $\widehat{V}_i$ are either both recurrent or both transient.

Next, let $Z_i = \widehat{V}_{i+1} - F^{(i)}_M$.
We claim that $Z_i$ is recurrent if and only if $\widehat{V}_i$ is also recurrent. 
Since $0\leq Z_i \leq \widehat{V}_{i+1}$, $Z_i$ is recurrent if  $\widehat{V}_i$ is recurrent. To see the other implication, note that $Z_i = F^{(i)}_{(\widehat{V}_i + 1)\vee M} - F^{(i)}_M$ is the number of failures in $\{\xi_{i,j}\}_{j\geq 1}$  between the $M$-th success and success number $(\widehat{V}_i + 1)\vee M$. Therefore, $Z_i$ is independent of $F^{(i)}_M$. Since $F^{(i)}_M$ is an i.i.d.\ sequence, then
\[
 \sum_{i\geq 0} P\left( \widehat{V}_{i+1} = 0 \right) \geq \sum_{i\geq 0} P\left(Z_i=0, \, F_M^{(i)} = 0\right) = P\left(F_M^{(0)} = 0\right) \sum_{i\geq 0} P(Z_i = 0), 
\]
and thus $\widehat{V}_{i}$ is recurrent if $Z_i$ is recurrent.

Finally, it can be shown that $Z_i$ is a branching process with migration where the migration component has mean $1-\d$ and the branching component has offspring distribution that is Geometric($1/2$) (see Lemmas 16 and 17 in \cite{kzPNERW}). Then, previous results in the branching process with migration literature show that $Z_i$ is recurrent if and only if $\d \geq 0$ (see Theorem A and Corollary 4 in \cite{kzPNERW} for a summary of these results). 
\end{proof}

We close this section by noting that the above regeneration structure for the process $V_i$ can be used to give a representation for the limiting speed of the excited random walk.  
First note that, as was shown in \cite{bsCRWspeed}, the representation \eqref{TnVirep} can be used to show that when $\d>1$, 
\[
  \frac{1}{v_0} = \lim_{n\ra\infty} \frac{T_n}{n} = 1 + 2 \lim_{n\ra\infty} \frac{1}{n} \sum_{i=1}^n V_i.
\]
To compute the last limit above, first note that $\{ (W_k-W_{k-1}, \s_{k}-\s_{k-1}) \}_{k\geq 1}$ is an i.i.d.\ sequence and that 
the tail decay of $\s_1$ given in Theorem 2.1 of \cite{kmLLCRW} implies that $E[\s_1] < \infty$ whenever $\d>1$. 
Let $k(n)$ be defined by $\s_{k(n)-1} < n \leq \s_{k(n)}$. A standard renewal theory argument implies that 
\[
 \lim_{n\ra\infty} \frac{k(n)}{n} = \frac{1}{E[\s_1]}. 
\]
Since $ W_{k(n)-1} \leq \sum_{i=1}^n V_n \leq W_{k(n)} $
and $\lim_{k\ra\infty} W_k/k = E[W_1]$, this implies that
\[
 \lim_{n\ra\infty} \frac{1}{n} \sum_{i=1}^n V_i =  \frac{E[W_1]}{E[\s_1]}. 
\]
Therefore, we obtain the following formula for the limiting speed of transient excited random walks. 
\begin{lem}
If $\d>1$, then
\begin{equation}\label{speedformula}
 v_0 = \frac{\E[\s_1]}{\E[\s_1 + 2 W_1]}.
\end{equation}
\end{lem}
\begin{rem}
 The tail decay of $W_1$ in \eqref{Wstails} implies that $E[W_1]=\infty$ when $\d \in (1,2]$. However, the limiting speed $v_0 = 0$ when $\d \in (1,2]$ so that the equality \eqref{speedformula} still holds in this case. 
\end{rem}

\section{Large Deviations for the Branching Process}\label{LDPVsec}

In this section we discuss the large deviations of $n^{-1} \sum_{i=1}^n V_i$. 
Let
\be\label{LWsdef}
 \Lambda_{W,\s}(\l,\eta) = \log \E\left[ e^{\l W_1 + \eta \s_1} \ind{\s_1 < \infty} \right] , \qquad \l,\eta \in \R,
\ee
be the logarithmic moment generating function of $(W_1,\s_1)$, and let 
\be\label{LVIVdef}
 \L_V(\l) = - \sup \{ \eta : \L_{W,\s}(\l,\eta) \leq 0 \} \quad \text{and} \quad I_V(x) = \sup_\l \l x - \L_V(\l). 
\ee
The relevance of these functions is seen by the following Theorem, which is a direct application of a more general result of Nummelin and Ney (see remark (ii) at the bottom of page 594 in \cite{nnMAP2}).


\begin{thm}\label{LDPVn0} 
Let $I_V(x)$ be defined as in \eqref{LVIVdef}. Then,
\[
 \liminf_{n\ra\infty} \frac{1}{n} \log P\left( \frac{1}{n} \sum_{i=1}^n V_i \in G, \, V_n = j \right) \geq - \inf_{x \in G} I_V(x),
\]
for all open $G$ and any $j\geq 0$, and 
\[
 \limsup_{n\ra\infty} \frac{1}{n} \log P\left( \frac{1}{n} \sum_{i=1}^n V_i \in F, \, V_n = j \right) \leq - \inf_{x \in F} I_V(x),
\]
for all closed $F$ and any $j\geq 0$. 
\end{thm}

In order to obtain large deviation results for the related excited random walk, it will also be necessary to obtain large deviation 
asymptotics of $n^{-1} \sum_{i=1}^n V_i$ without the added condition on the value of $V_n$. 

\begin{thm}\label{LDPVn}
 Let $I_V(x)$ be defined as in \eqref{LVIVdef}. Then, 
$n^{-1} \sum_{i=1}^n V_i$ satisfies a large deviation principle with rate function $I_V(x)$. That is, 
\[
 \liminf_{n\ra\infty} \frac{1}{n} \log P\left( \frac{1}{n} \sum_{i=1}^n V_i \in G \right) \geq - \inf_{x \in G} I_V(x),
\]
for all open $G$, and 
\be\label{LDPVnub}
 \limsup_{n\ra\infty} \frac{1}{n} \log P\left( \frac{1}{n} \sum_{i=1}^n V_i \in F \right) \leq - \inf_{x \in F} I_V(x),
\ee
for all closed $F$. 
\end{thm}

\begin{rem}
 There are many results in the large deviations literature that imply a large deviation principle for the empirical mean of a Markov chain. However, we were not able to find a suitable theorem that implied Theorem \ref{LDPVn}. Some of the existing results required some sort of fast mixing of the Markov chain \cite{bdLDPMixing, dzLDTA}, but the Markov chain $\{V_i\}_{i\geq 0}$ mixes very slowly since if $V_0$ is large it typically takes a long time to return to $0$ (on the order of $\bigo(V_0)$ steps). Moreover, it is very important that the rate functions are the same in Theorems \ref{LDPVn0} and \ref{LDPVn}, and many of the results for the large deviations for the empirical mean of a Markov chain formulate the rate function in terms of the spectral radius of an operator \cite{daUBLD} instead of in terms of logarithmic moment generating functions as in \eqref{LWsdef} and \eqref{LVIVdef}. 
\end{rem}
\begin{proof}
Obviously the lower bound in Theorem \ref{LDPVn} follows from the corresponding lower bound in \eqref{LDPVn0}, and so it is enough to prove the upper bound only. 
Our proof will use the following facts about the functions $\L_V$ and $I_V$. 
\begin{enumerate}
 \item $\L_V(\l)$ is convex and continuous on $(-\infty,0]$ and $\L_V(\l) = \infty$ for all $\l>0$. Therefore, $I_V(x) = \sup_{\l<0} \left(\l x - \L_V(\l)\right)$. \label{IVprop1}
 \item $I_V(x)$ is a convex, non-increasing function of $x$, and $\lim_{x\ra\infty} I_V(x) = \inf_x I_V(x) = 0$. \label{IVprop2}
\end{enumerate}
These properties and more will be shown in Section \ref{IVprop} below where we give a qualitative description of the rate function $I_V$.  
By property \ref{IVprop2} above, 
it will be enough to prove the large deviation upper bound for closed sets of the form $F=(-\infty,x]$. That is, we 
 need only to show that
\be\label{ltub}
 \limsup_{n\ra\infty} \frac{1}{n} \log P\left( \sum_{i=1}^n V_i \leq x n \right) \leq -I_V(x), \quad \forall x <\infty. 
\ee
This will follow from 
\be\label{lglm}
 \limsup_{n\ra\infty} \frac{1}{n} \log E[ e^{\l \sum_{i=1}^n V_i} ] \leq \L_V(\l), \quad \forall \l < 0. 
\ee
Indeed, combining \eqref{lglm} with the usual Chebyshev upper bound for large deviations gives that for any $x<\infty$ and $\l<0$,
\[
 \limsup_{n\ra\infty} \frac{1}{n} \log P\left( \sum_{i=1}^n V_i \leq x n \right) 
\leq  \limsup_{n\ra\infty} \frac{1}{n} \log ( e^{-\l x n} E[ e^{\l \sum_{i=1}^n V_i} ] ) \leq - \l x + \L_V(\l).
\]
Optimizing over $\l<0$ and using property \ref{IVprop1} above proves \eqref{ltub}.

It remains still to prove \eqref{lglm}. 
By decomposing according to the time of the last regeneration before $n$ we obtain
\begin{align}
& E[ e^{ \l \sum_{i=1}^n V_i - \L_V(\l) n} ] \nonumber \\
&= E[ e^{\l \sum_{i=1}^n V_i- \L_V(\l) n} \ind{ \s_1 > n } ] + \sum_{m=1}^n \sum_{t=0}^{n-1} E[ e^{ \l \sum_{i=1}^n V_i - \L_V(\l)n } \ind{ \s_m \leq n < \s_{m+1}, \, n-\s_m = t} ] \nonumber  \\
&= E[ e^{\l \sum_{i=1}^n V_i- \L_V(\l) n} \ind{ \s_1 > n } ] \nonumber \\
&\qquad + \sum_{m=1}^n \sum_{t=0}^{n-1} E[ e^{ \l W_m - \L_V(\l) \s_m} \ind{ \s_m = n-t} ] E[ e^{\l \sum_{i=1}^t V_i- \L_V(\l) t} \ind{ \s_1 > t } ] \nonumber \\
&\leq E[ e^{\l \sum_{i=1}^n V_i- \L_V(\l) n} \ind{ \s_1 > n } ] \label{Ek} \\
&\qquad + \left( \sum_{m=1}^n E[ e^{ \l W_m - \L_V(\l) \s_m} \ind{\s_m < \infty} ] \right) \left( \sum_{t=0}^{n-1}  E[ e^{\l \sum_{i=1}^t V_i- \L_V(\l) t} \ind{ \s_1 > t } ] \right) \label{twosums},
\end{align}
where we used the Markov property in the second equality. 
The definition of $\L_V$ and the monotone convergence theorem imply that $\L_{W,\s}(\l,-\L_V(\l)) \leq 0$. Therefore, 
\begin{align*}
 \sum_{m=1}^n E[ e^{ \l W_m - \L_V(\l) \s_m} \ind{\s_m < \infty} ] &= \sum_{m=1}^n e^{m \L_{W,\s}(\l,-\L_V(\l))} \leq n. 
\end{align*}

To bound the second sum in \eqref{twosums} we need the following lemma, whose proof we postpone for now. 
\begin{lem}\label{suptlem}
 For any $\l < 0$, 
\[
 \sup_{t\geq 0}  E[ e^{\l \sum_{i=1}^t V_i- \L_V(\l) t} \ind{ \s_1 > t } ] < \infty. 
\]
\end{lem}
Lemma \ref{suptlem} implies the expectation \eqref{Ek} is uniformly bounded in $n$ and that the second sum in \eqref{twosums} grows at most linearly in $n$.
Since the first sum in \eqref{twosums} also grows linearly in $n$ this implies that
\[
 \limsup_{n\ra\infty} \frac{1}{n} \log E[ e^{ \l \sum_{i=1}^n V_i - \L_V(\l) n} ] \leq 0, \quad \forall \l < 0, 
\]
which is obviously equivalent to \eqref{lglm}. It remains only to give the proof of Lemma \ref{suptlem}.

\begin{proof}[Proof of Lemma \ref{suptlem}]
First, note that
\begin{align}
1 \geq  E[e^{\l W_1 - \L_V(\l) \s_1} \ind{\s_1 < \infty} ] &\geq E[e^{\l W_1 - \L_V(\l) \s_1} \ind{t < \s_1 < \infty} ] \nonumber \\
&= E \left[ e^{\l \sum_{i=1}^t V_i - \L_V(\l) t} \ind{\s_1 > t} e^{\l \sum_{i=t+1}^{\s_1} V_i - \L_V(\l) (\s_1 -t) } \ind{\s_1 < \infty} \right] \nonumber \\
&= E \left[ e^{\l \sum_{i=1}^t V_i - \L_V(\l) t} \ind{\s_1 > t} E^{V_t}\left[ e^{\l W_1 - \L_V(\l) \s_1} \ind{\s_1 < \infty} \right] \right], \label{Vtcond}
\end{align}
where in the last equality we use the notation $E^m$ for the expectation with respect to the law of the Markov chain $V_i$ conditioned on $V_0 = m$. 
Since $V_i$ is an irreducible Markov chain and $E[e^{\l W_1 - \L_V(\l) \s_1} \ind{\s_1 < \infty} ] \leq 1$, then the inner expectation in \eqref{Vtcond} is finite for any value of $V_t$ and can be uniformly bounded below if $V_t$ is restricted to a finite set. 
Thus, for any $K<\infty$,
\begin{align}
 & E \left[ e^{\l \sum_{i=1}^t V_i - \L_V(\l) t} \ind{\s_1 > t, \, V_t \leq K} \right] \nonumber \\
&\quad \leq \left( \inf_{m \in [1,K]}  E^m[ e^{\l W_1 - \L_V(\l) \s_1} \ind{\s_1 < \infty} ]  \right)^{-1}  E[e^{\l W_1 - \L_V(\l) \s_1} \ind{\s_1 < \infty}].
\label{VtKbound}
\end{align}
Let $C_{K,\l} < \infty$ be defined to be the right side of the inequality above. 

Note that the upper bound \eqref{VtKbound} does not depend on $t$. 
The key to finishing the proof of Lemma \ref{suptlem} is using the upper bound \eqref{VtKbound} in an iterative way. 
For any $t\geq 1$, 
\begin{align*}
 E[ e^{\l \sum_{i=1}^t V_i- \L_V(\l) t} \ind{ \s_1 > t } ] &\leq C_{K,\l} + E[ e^{\l \sum_{i=1}^t V_i- \L_V(\l) t} \ind{ \s_1 > t, \, V_t > K } ] \\
& \leq C_{K,\l} + e^{\l K - \L_V(\l)} E [ e^{\l \sum_{i=1}^{t-1} V_i- \L_V(\l) (t-1)} \ind{ \s_1 > t-1 } ],
\end{align*}
where in the last inequality we used that $\{ \s_1 > t, \, V_t > K \} = \{ \s_1 > t-1, \, V_t > K \}$.
Iterating the above bound implies that 
\[
 E[ e^{\l \sum_{i=1}^t V_i- \L_V(\l) t} \ind{ \s_1 > t } ] \leq C_{K,\l} \sum_{l=0}^{t-1} e^{l(\l K - \L_V(\l))} + e^{t(\l K - \L_V(\l))}.
\]
By choosing $K > \L_V(\l)/\l$ so that $e^{\l K - \L_V(\l)} < 1$, we thus obtain that 
\[
 \sup_{t\geq 0} E [ e^{\l \sum_{i=1}^t V_i- \L_V(\l) t} \ind{ \s_1 > t } ] \leq \frac{C_{K,\l}}{1-e^{K \l - \L_V(\l)}} + 1 < \infty. 
\]
\end{proof}
\end{proof}


\subsection{Properties of the rate function $I_V$}\label{IVprop}

We now turn our attention to a qualitative description of the rate function $I_V$. Since $I_V$ is defined as the Legendre dual of $\L_V$, these properties will in turn follow from an understanding of $\L_V$ (and also $\L_{W,\s}$). 
We begin with some very basic properties of $\L_V$ and the corresponding properties of $I_V$. 

\begin{lem}\label{infdomain}
 $\L_V(\l)$ is non-decreasing, convex, and left-continuous as a function of $\l$. 
Moreover,
\begin{enumerate}
 \item $\L_V(\l) \in (\log \Ev[\w_0(1)], 0)$ for all $\l < 0$, and $\lim_{\l \ra-\infty} \L_V(\l) = \log \Ev[\w_0(1)]$. \label{LVlbounds}
 \item $\L_V(\l) = \infty$ if $\l > 0$.
\end{enumerate}
\end{lem}
\begin{proof}
Recall the definitions of $\L_{W,\s}$ and $\L_V$ in \eqref{LWsdef} and \eqref{LVIVdef}, respectively. 
The fact that $\L_V(\l)$ is non-decreasing follows from the fact that $\L_{W,\s}(\l_1,\eta) \leq \L_{W,\s}(\l_2,\eta)$ for any $\l_1 < \l_2$. Since $\L_{W,\s}$ is the logarithmic generating function of the joint random variables $(W_1,\s_1)$, then $\L_{W,\s}(\l,\eta)$ is a convex function of $(\l,\eta)$ (and strictly convex on $\{(\l,\eta): \, \L_{W,\s}(\l,\eta) < \infty \}$). 
The convexity of $\L_V$ as a function of $\l$ then follows easily from the convexity of $\L_{W,\s}(\l,\eta)$ and the definition of $\L_V$. 
Also, left-continuity follows from the definition of $\L_V$ and the fact that $\lim_{\l\ra \l'} \L_{W,\s}(\l,\eta) = \L_{W,\s}(\l',\eta)$ by the monotone convergence theorem.

Since $W_1 \geq 0$,  $\L_{W,\s}(\l,0) = \log E[ e^{\l W_1} \ind{\s_1 < \infty}] < 0$ for all $\l < 0$. On the other hand, since $W_1 \geq \s_1 - 1$ it follows that $\L_{W,\s}(\l, -\l) \leq -\l < \infty$ for all $\l \leq 0$. Then the continuity of $\L_{W,\s}$ and the definition of $\L_V(\l)$ imply that $\L_V(\l) < 0$ for all $\l<0$. 
Additionally, 
\[
 E[ e^{\l W_1 + \eta \s_1} \ind{\s_1 < \infty} ] > e^\eta P(\s_1 = 1) = e^{\eta} \Ev[ \w_0(1) ], 
\]
which implies that $\L_{W,\s}(\l, - \log \Ev[ \w_0(1)] ) > 0$ for all $\l<0$. 
Thus, $\L_V(\l) > \log \Ev[ \w_0(1)]$ for all $\l < 0$. 
To prove the second part of property \ref{LVlbounds}, note that $\s_1 \leq W_1 + 1$ implies that for $\eta \geq 0$,
\be\label{LWslneg}
 \lim_{\l \ra -\infty} E[ e^{\l W_1 + \eta \s_1} \ind{\s_1 < \infty} ] \leq \lim_{\l \ra -\infty} E[e^{(\l + \eta)W_1 + \eta} ] = e^\eta P(W_1 = 0) = e^\eta \Ev[ \w_0(1)], 
\ee
where the second to last equality follows from the bounded convergence theorem. From \eqref{LWslneg} and the definition of $\L_V$, it follows that $\lim_{\l \ra -\infty} \L_V(\l) \leq \log \Ev[\w_0(1)]$. Combining this with the first part of property \ref{LVlbounds} implies the second part of property \ref{LVlbounds}.

To show that $\L_V(\l) = \infty$ for $\l>0$ it is actually easiest to refer back to the excited random walk. 
Recall the naive strategy for slowdowns of the excited random walk in Example \ref{slowdownex}. 
We can modify the strategy slightly by not only consuming all cookies in $[0,n^{1/3})$ and then staying in the interval for $n$ steps, but also requiring that the random walk then exits the interval on the right. This event still has a probability bounded below by $C e^{-c n^{1/3}}$. 
Examining the branching process corresponding to the excited random walk we see that the event for the random walk described above implies that 
$U_i^{N} \geq 1$ for all $i\in [1,N-1]$, $U_0^{N} = 0$ and $\sum_{i=1}^{N} U_i^{N} > n/2$, where $N= \lceil n^{1/3} \rceil$. 
Then, using \eqref{UiVirep} we obtain that that $ P( W_1 > n/2, \, \s_1 =  \lceil n^{1/3} \rceil ) \geq C e^{-c n^{1/3} }$ for all $n\geq 1$ which implies that
\[
 E[e^{\l W_1 + \eta \s_1} \ind{\s_1 < \infty} ] \geq e^{\l n/2 + \eta n^{1/3}} P(W_1 > n/2, \, \s_1 =  \lceil n^{1/3} \rceil ) \geq C e^{ \l n/2 + \eta n^{1/3} -c n^{1/3} }, 
\]
for any $\l > 0$ and $\eta < 0$. Since this lower bound can be made arbitrarily large by taking $n\ra\infty$, this shows that $\L_{W,\s}(\l,\eta) = \infty$ for any $\l>0$ and $\eta< 0$, and thus $\L_V(\l) = \infty$ for all $\l>0$.  
\end{proof}

We would like to say that $\L_{W,\s}(\l,-\L_V(\l)) = 0$. However, in order to be able to conclude this is true, we need to show that $\L_{W,\s}(\l, \eta) \in [ 0 , \infty)$ for some $\eta$. 
The next series of lemmas gives some conditions where we can conclude this is true. 

\begin{lem}\label{lveryneg}
If $\l \leq \log \Ev[\w_0(1)]$, then 
\be\label{impsoln}
 \L_{W,\s}(\l,-\L_V(\l)) = 0.
\ee
Moreover, $\L_V(\l)$ is strictly convex and analytic on $(-\infty, \log \Ev[\w_0(1)])$. 
\end{lem}
\begin{proof}
 Since $W_1 \geq \s_1 - 1$ we have that for $\l\leq 0$,
\[
 \L_{W,\s}(\l,-\l) = \log E [e^{\l(W_1 - \s_1)} \ind{\s_1 < \infty}] \leq -\l. 
\]
 Therefore, $\L_{W,\s}(\l,\eta) < \infty$ for all $\l<0$ and $\eta \leq -\l$. On the other hand, it was shown above that $\L_{W,\s}(\l,0) < 0$ and $\L_{W,\s}(\l, -\log \Ev[ \w_0(1)]) > 0$ when $\l < 0$. Since $\L_{W,\s}(\l,\eta)$ is monotone increasing and continuous in $\eta$ this implies that $\L_{W,\s}(\l,\eta) = 0$ has a unique solution $\eta \in [0,-\log \Ev[\w_0(1)]]$ when $\l \leq \log \Ev[\w_0(1)]$. By the definition of $\L_V$ and the fact that $\L_{W,\s}(\l,\eta)$ is strictly increasing in $\eta$, this must be $\eta = -\L_V(\l)$.  

Let $\mathcal{D}_{W,\s} = \{(\l,\eta): \, \L_{W,\s}(\l,\eta) < \infty\}$. 
The above argument shows not only that \eqref{impsoln} holds but also that $(\l,-\L_V(\l))$ is in the interior of $\mathcal{D}_{W,\s}$ when $\l \leq \log \Ev[\w_0(1)]$.
Since $\L_{W,\s}$ is analytic on $\mathcal{D}_{W,\s}$, the implicit function theorem implies that $\L_V(\l)$ is analytic on $(-\infty,\log \Ev[\w_0(1)])$. Finally, combining \eqref{impsoln} with the fact that $\L_{W,\s}$ is strictly convex on $ \mathcal{D}_{W,\s}$ implies that $\L_V(\l)$ is strictly convex on $(-\infty,\log \Ev[\w_0(1)])$.
\end{proof}

\begin{lem}\label{slope}
 For every $m < \infty$, there exists a $\l_0=\l_0(m) < 0$ such that $\L_{W,\s}(\l,-\l m) < \infty$ for all $\l \in (\l_0, 0)$. 
\end{lem}
\begin{proof}
 We need to show that $E[e^{\l W_1 - \l m \s_1} \ind{\s_1 < \infty}] = E[e^{\l\sum_{i=1}^{\s_1}(V_i - m) }\ind{\s_1 < \infty}] < \infty$ for $\l$ negative and sufficiently close to zero. Since $\l<0$ we need to bound the sum in the exponent from below. Note that all the terms in the sum except the last one are larger than $-(m-1)$ and that the terms are non-negative if $V_i \geq m$. Therefore, letting $N_m = \#\{ 1\leq i\leq \s_1 : \, V_i < m \}$ we obtain that 
\[
 E[e^{\l W_1 - \l m \s_1} \ind{\s_1 < \infty}] \leq E[ e^{- \l (m-1) N_m } ]. 
\]
To show that this last expectation is finite for $\l$ close to zero, we need to show that $N_m$ has exponential tails. To this end, note that the event $\{N_m > n \}$ implies that the first $n$ times that the process $V_i < m$, the following step is not to $0$. Thus, 
\[
 P( N_m > n ) \leq \left( \max_{k < m} P(V_1 \neq 0 \, | \, V_0 = k ) \right)^n = P(V_1 \neq 0 \, | \, V_0 = \lceil m \rceil -1 )^n.
\]
Therefore, the statement of the Lemma holds with 
\[
 \l_0(m) = \frac{1}{m-1} \log  P(V_1 \neq 0 \, | \, V_0 = \lceil m \rceil -1 ).
\]
\end{proof}

\begin{cor}\label{LVnear0}
 If $\d>2$ (so that $E[W_1], E[\s_1] < \infty$), then there exists a $\l_1 < 0$ such that on the interval $(\l_1,0)$
\begin{enumerate}
 \item $\L_{W,\s}(\l,-\L_V(\l)) = 0$.
 \item $\L_V(\l)$ is analytic and strictly convex as a function of $\l$. 
 \item $\lim_{\l\ra 0^+} \L_V'(\l) = E[W_1]/E[\s_1] =: m_0$. 
\end{enumerate}
\end{cor}
\begin{proof}
 Let $\mathcal{D}_{W,\s} = \{ (\l,\eta) : \, \L_{W,\s}(\l,\eta) < \infty \}$ be the domain where $\L_{W,\s}$ is finite, and let $D_{W,\s}^\circ$ be the interior of $D_{W,\s}$. 
Define $m_0 = E[W_1]/E[\s_1]$. 
Then, if $0 > \l > \l_0(m) m/m_0$ for some $m>m_0$, it follows that 
\[
 \L_{W,\s}(\l,-\l m_0) \leq \L_{W,\s}\left(\frac{\l m_0}{m}, -\l m_0 \right) < \infty,
\]
where the first inequality follows from the monotonicity of $\L_{W,\s}$ and the last inequality follows from Lemma \ref{slope}.
Thus, $(\l,-\l m_0) \in \mathcal{D}_{W,\s}^{\circ}$
if $\l \in (\l_1, 0)$ with $\l_1 = \inf_{m>m_0} \l_0(m) m/m_0$. 

Since $\L_{W,\s}$ is analytic and strictly convex in $D_{W,\s}^\circ$, the function $g(\l) = \L_{W,\s}(\l,-m_0\l)$ is strictly convex and analytic on the interval $(\l_1,0)$. 
In particular, $g$ is differentiable and 
\[
 g'(\l) = \frac{d}{d \l} \log E\left[ e^{\l(W_1 -m_0 \s_1)} \right] = \frac{E\left[ (W_1-m_0 \s_1) e^{\l(W_1 -m_0 \s_1)} \right]}{E\left[ e^{\l(W_1 -m_0 \s_1)} \right]}.
\]
Since $g$ is strictly convex, 
\[
 g'(\l) < \lim_{\l\ra 0^-} g'(\l) = E[ W_1 - m_0 \s_1 ] = 0, \qquad \forall \l \in (\l_1,0). 
\]
Therefore, $g(\l)$ is strictly decreasing on $(\l_1,0)$. Since, $\lim_{\l\ra 0^-} g(\l) = g(0) = 0$ we obtain that $g(\l) = \L_{W,\s}(\l,-m_0 \l) > 0$ for $\l \in (\l_1,0)$. 
Thus, for every $\l \in (\l_1,0)$ there exists an $\eta \in (0, -m_0 \l)$ such that $\L_{W,\s}(\l,\eta) = 0$, and the definition of $\L_V$ implies that  $\eta = -\L_V(\l)$. 
We have shown that $\L_{W,\s}(\l,-\L_V(\l)) = 0$ and $(\l,-\L_V(\l)) \in \mathcal{D}_{W,\s}^\circ$ for all $\l\in(\l_1,0)$. As was the case in the proof of Lemma \ref{lveryneg} these facts imply that $\L_V(\l)$ is analytic and strictly convex on $(\l_1,0)$. 

To show that $\lim_{\l\ra 0^-} \L_V'(\l) = m_0$, first note that as was shown above $\L_V(\l) > m_0 \l$ for $\l<0$. For $m<m_0$ define $g_m(\l) = \L_{W,\s}(\l,-m \l)$. For $\l$ close enough to $0$ we have that $g_m(\l)$ is strictly convex and analytic, and that
\[
 g_m'(\l) = \frac{d}{d \l} \log E\left[ e^{\l(W_1 -m \s_1)} \right] = \frac{E\left[ (W_1-m \s_1) e^{\l(W_1 -m \s_1)} \right]}{E\left[ e^{\l(W_1 -m \s_1)} \right]}.
\]
Therefore, $\lim_{\l\ra 0^-} g_m'(\l) = E[W_1]-m E[\s_1] > 0$, and thus there exists a $\l_2=\l_2(m)<0$ such that $g_m(\l) = \L_{W,\s}(\l,-m\l) < 0$ for $\l\in(\l_2,0)$. This implies that $m_0 \l < \L_V(\l) < m \l$ for all $\l \in (\l_2,0)$, and thus $\lim_{\l\ra 0^-} \L_V'(\l) \in [m,m_0]$. The proof is finished by noting that this is true for any $m<m_0$. 
\end{proof}

We are now ready to deduce some properties of the rate function $I_V$. 
\begin{lem}\label{infIVlem}
$\inf_x I_V(x) = 0$.
\end{lem}
\begin{proof}
Since $I_V$ is the Legendre transform of $\L_V$ and $\L_V$ is lower-semicontinuous, then it follows that $\inf_x I_V(x) = - \L_V(0)$. 
If $\d\geq 0$, then $\L_{W,\s}(0,0) = \log P(\s_1<\infty) = 0$ (by Lemma \ref{Vrt}) and thus $\L_V(0) = 0$ when $\d\geq 0$. 

If $\d < 0$, then $\L_{W,\s}(0,0) = \log P(\s_1<\infty) < 0$ and so we can only conclude a priori\footnote{Note that we cannot conclude that $\L_V(0) < 0$ since we do not know if $\L_{W,\s}(0,\eta)<\infty$ for some $\eta>0$. In fact since $\inf_x I_V(x) = 0$ if and only if $\L_V(0) = 0$, the proof of the lemma shows indirectly that $\L_{W,\s}(0,\eta) = \infty$ for all $\eta > 0$ when $\d <0$.}
that $\L_V(0) \leq 0$.
Instead we will prove $\inf_x I_V(x) = 0$ in a completely different manner. 
First note that letting $F=G=\R$ in Theorem \ref{LDPVn} implies that 
\[
 \lim_{n\ra\infty} \frac{1}{n} \log P(V_n = 0) = -\inf_x I_V(x). 
\]
Therefore, we need to show that $P(V_n = 0)$ does not decay exponentially fast in $n$. The explanation of the representation \eqref{TnVirep} implies that $P(T_n < T_{-1}) = P(U_0^n = 0) \leq P(V_n = 0)$,
and thus we are reduced to showing that $P(T_n < T_{-1})$ does not decay exponentially fast in $n$. 
In fact, we claim that there exists a constant $C>0$ such that
\be\label{TnT1}
 P(T_n < T_{-1}) \geq C n^{-M-1} 
\ee
To see this, suppose that the first $2M+1$ steps of the random walk alternate between 0 and 1. 
The probability of this happening is
\begin{align*}
 P(X_{2i}=0, \, X_{2i+1}=1,\, i=1,2,\ldots M) &= \Ev\left[ \prod_{j=1}^{M+1} \w_0(j)\prod_{j=1}^M (1-\w_1(j)) \right] \\
&= \frac{1}{2} \Ev\left[ \prod_{j=1}^M \w_0(j) \right] \Ev\left[ \prod_{j=1}^M (1- \w_0(j)) \right] > 0.
\end{align*}
At this point the random walker has consumed all the ``cookies'' at the sites $0$ and $1$. Therefore, by a simple symmetric random walk computation, the probability that the random walk from this point hits $x=2$ before $x=-1$ is $2/3$. Since $\d<0$ the random walk will eventually return from $x=2$ to $x=1$ again, and then the probability that the random walk again jumps $M$ more times from $x=1$ to $x=2$ without hitting $x=-1$ is $(2/3)^{M}$. After jumping from $x=1$ to $x=2$ a total of $M+1$ times there are no longer any cookies at $x=2$ either, and thus the probability that the random walk now jumps $M+1$ times from $x=2$ to $x=3$ without visiting $x=-1$ is $(3/4)^{M+1}$. We continue this process at successive sites to the right until the random walk makes $M+1$ jumps from $x=n-2$ to $x=n-1$ without hitting $x=-1$ (which happens with probability $((n-1)/n)^{M+1}$). Upon this last jump to $x=n-1$ the random walk has consumed all cookies at $x=n-1$ and so the probability that the next step is to the right is $1/2$. Putting together the above information we obtain the lower bound
\[
 P(T_n < T_{-1}) \geq  \left( \frac{1}{2} \Ev\left[ \prod_{j=1}^M \w_0(j) \right] \Ev\left[ \prod_{j=1}^M (1- \w_0(j)) \right] \right) \left( \frac{2}{3} \frac{3}{4} \cdots \frac{n-1}{n} \right)^{M+1} \frac{1}{2}. 
\]
This completes the proof of \eqref{TnT1}, and thus $\inf_x I_V(x) = 0$ when $\d<0$. 
\end{proof}

\begin{lem}\label{IVprops}
 The function $I_V(x)$ is convex, non-increasing, and continuous on $[0,\infty)$. Moreover,
\begin{enumerate}
 \item There exists an $m_2>0$ such that $I_V(x)$ is strictly convex and analytic on $(0,m_2)$. 
 \item $I_V(0) = -\log \Ev[\w_0(1)]$ and $\lim_{x \ra 0^+} I_V'(x) = -\infty$. 
 \item If $\d>2$ then there exists an $m_1< m_0 = E[W_1]/E[\s_1]$ such that $I_V(x)$ is strictly convex and analytic on $(m_1,m_0)$, $I_V(x) = 0$ for $x\geq m_0$, and $\lim_{x \ra m_0^-} I_V'(x) = 0$ so that $I_V$ is continuously differentiable on $(m_1,\infty)$. 
 \item If $\d\leq 2$ then $I_V(x) > 0$ for all $x < \infty$. 
\end{enumerate}
\end{lem}
\begin{proof}
Since $I_V$ is the Legendre transform of $\L_V$, $I_V(x)$ is convex and lower-semicontinuous as a function in $x$.  
It follows easily from Lemma \ref{infdomain} and the definition of $I_V$ that $I_V(x)<\infty$ if and only if $x \in[0,\infty)$, 
and since $I_V$ is convex and lower-semicontinuous this shows that $I_V$ is continuous on $[0,\infty)$. 
The fact that $I_V(x)$ is non-increasing follows from the fact that $\L_V(\l) = \infty$ for any $\l > 0$. Indeed, if $x_1 \leq x_2$ then 
\[
 I_V(x_1) = \sup_{\l\leq 0} \left\{ \l x_1 - \L_V(\l) \right\} \geq \sup_{\l\leq 0} \left\{ \l x_2 - \L_V(\l) \right\} = I_V(x_2). 
\]

Next, recall from Lemma \ref{lveryneg} that $\L_V(\l)$ is strictly convex and analytic on $(-\infty,\log \Ev[\w_0(1)])$ and let $m_2 = \lim_{\l \ra \log \Ev[\w_0(1)])^-} \L_V'(\l)$. 
The fact that $\L_V(\l)$ is non-decreasing and uniformly bounded below also implies that $\lim_{\l \ra -\infty} \L_V'(\l) = 0$. Therefore, for every $x \in (0,m_2)$ there exists a unique $\l = \l(x)$ such that $\L_V'(\l(x)) = x$ and so 
\be\label{IVform}
I_V(x) = \l(x) x - \L_V(\l(x)) \quad \text{for } x \in (0,m_2). 
\ee
Since $\L_V(\l)$ is analytic on $(-\infty, \log \Ev[\w_0(1)])$ the inverse function theorem implies that $\l(x)$ is analytic on $(0,m_2)$ and thus \eqref{IVform} implies that $I_V(x)$ is  analytic on $(0,m_2)$ as well.  
To see that $I_V(x)$ is strictly convex on $(0,m_2)$, we differentiate \eqref{IVform} with respect to $x$ and use the fact that $\L'_V(\l(x)) = x$ for $x \in (0,m_2)$ to obtain  
\be\label{IVderiv}
 I'_V(x) = \l(x), \quad \text{for } x \in (0,m_2). 
\ee
Since $\l(x)$ is strictly increasing on $(0,m_2)$, it follows that $I_V$ is strictly convex on $(0,m_2)$. 
Moreover, \eqref{IVderiv} implies that $\lim_{x \ra 0^+} I'_V(x) = \lim_{x \ra 0^+} \l(x) = -\infty$ and Lemma \ref{infdomain} \ref{LVlbounds} implies that that $I_V(0) = -\inf_\l \L_V(\l) = -\log \Ev[\w_0(1)]$.

When $\d>2$, Lemma \ref{LVnear0} implies that $\L_V(\l)$ is analytic and strictly convex on $(\l_1,0)$. Let $m_1 = \lim_{\l \ra \l_1^+} \L_V'(\l)$ and recall that $\lim_{\l \ra 0^-} \L_V'(\l) = m_0 = E[W_1]/E[\s_1]$. Then the same argument as above shows that $I_V(x)$ is strictly convex and analytic on $(m_1,m_0)$ and that $\lim_{x \ra m_0^-} I_V'(x) = 0$. Now, since $\L_V'(\l)$ increases to $m_0$ as $\l \ra 0^-$ and $\L_V(0) = 0$, then $\L_V(\l) \geq m_0 \l$ for all $\l\leq 0$. Therefore 
\[
 I_V(x) = \sup_{\l \leq 0} \left\{ \l x - \L_V(\l) \right\} \leq \sup_{\l \leq 0} \left\{ \l m_0 - \L_V(\l) \right\} = 0, \quad \text{ for all } x \geq m_0, 
\]
where the first equality follows from the fact that $\L_V(\l) = \infty$ if $\l>0$. However, since $I_V(x) \geq -\L_V(0) = 0$ it must be that $I_V(x) = 0$ for $x \geq m_0$.

It remains only to show that $I_V(x)>0$ for all $x$ when $\d\leq 2$. We will divide the proof into two cases: $\d \in (1,2]$ and $\d \leq 1$. 

\noindent\textbf{Case I:} $\d \in (1,2]$.\\
For any $m<\infty$ let $g_m(\l) = \L_{W,\s}(\l, -m \l)$. Then, as in the proof of Corollary \ref{LVnear0}, $g_m(\l)$ is analytic and strictly convex for $\l<0$ close enough to zero. Moreover, 
\[
 \lim_{\l \ra 0^-} g_m'(\l) = E[W_1 - m \s_1] = \infty,
\]
where the last equality holds since the tail decay of $W_1$ and $\s_1$ in \eqref{Wstails} implies that  $E[W_1] = \infty$ and $E[\s_1] < \infty$ when $\d \in (1,2]$.
Since $g_m(0) = \L_{W,\s}(0,0) = 0$ this implies that $g_m(\l) = \L_{W,\s}(\l, -m \l) < 0$ for $\l<0$ sufficiently close to 0, and therefore $\limsup_{\l \ra 0^-} \L_V(\l)/\l \geq m $. Since this is true for any $m<\infty$ and since $\L_V(\l)$ is convex, 
it follows that $\lim_{\l \ra 0^-} \L_V(\l)/\l = \infty$. 
Thus, for any $x<\infty$ there exists a $\l'<0$ such that $\L_V(\l') < \l' x$ and so 
$I_V(x) \geq \l' x - \L_V(\l') > 0$. 

\noindent\textbf{Case II:} $\d \leq 1$. \\
As in the case $\d \in (1,2]$ we could proceed by arguing that $g_m'(\l) \ra E[(W_1 - m\s_1)\ind{\s_1<\infty}]$. However, we would need to then show that this last expectation is infinite, and this would require an analysis of the joint tail behavior of $(W_1,\s_1)$.
This could probably be achieved in the case $\d \in (0,1)$ by adapting the arguments of Kosygina and Mountford in \cite{kmLLCRW}, however when $\d<0$ it would be more difficult since in that case the Markov chain is transient and we would need to analyze the tails of $\s_1$ conditioned on $\s_1<\infty$. 
 It is possible that such an approach would work, but we will give a softer argument instead. 

Let $\L_1(\l) = \limsup_{n\ra\infty} \frac{1}{n} \log E[ e^{\l \sum_{i=1}^n V_i } ]$.
Then, the standard Chebyshev large deviation upper bound implies that 
\[
 \limsup_{n\ra\infty} \frac{1}{n} \log P\left( \sum_{i=1}^n V_i < xn \right) \leq - \sup_{\l < 0} ( \l x - \L_1(\l)). 
\]
On the other hand, Theorem \ref{LDPVn0} and the fact that $I_V$ is non-increasing implies that 
\[
 \lim_{n\ra\infty} \frac{1}{n} \log P\left( \sum_{i=1}^n V_i < xn , \, V_n = 0 \right) = -I_V(x). 
\]
Thus, we see that $I_V(x) \geq  \sup_{\l < 0} ( \l x - \L_1(\l))$ for any $x<\infty$. 
Then, similar to the case $\d \in (1,2]$ above, it will follow that $I_V(x)>0$ for all $x<\infty$ if we can show that $\lim_{\l\ra 0^-} \L_1(\l)/\l = \infty$.

Fix an integer $K \geq 1$. If $\l<0$, then $\l \sum_{i=1}^n V_i \leq \l K \sum_{i=1}^n \ind{V_i \geq K }$. Thus, 
\be\label{Kdec}
 E[e^{\l \sum_{i=1}^n V_i} ] \leq e^{\l K (1-\theta)n} + P\left( \sum_{i=1}^n \ind{V_i < K} > \theta n \right). 
\ee
Recall the construction of the process $V_i$ in Section \ref{Visec} and define $\tilde{V}_i$ by $\tilde{V}_0 = 0$ and 
\[
 \tilde{V}_{i+1} = F^{(i)}_{\tilde{V}_{i} + 1} \ind{F^{(i)}_{\tilde{V}_{i} + 1} \geq K}.  
\]
That is, jumps are governed by the same process as the jumps of the Markov chain $V_i$ with the exception that any attempted jump to a site in $[1,K-1]$ is replaced by a jump to 0. 
Note that the above construction of $\tilde{V}_i$ gives a natural coupling with $V_i$ so that $\tilde{V}_i \leq V_i$ for all $i$. Let $\tilde{\s}_k$, $k=1,2,\ldots$ be the successive return times to $0$ of the Markov chain $\tilde{V}_i$. Then, since $\tilde{V}_i < K$ implies that $\tilde{V}_i = 0$, 
\[
 P\left( \sum_{i=1}^n \ind{V_i < K} > \theta n \right) \leq P\left( \sum_{i=1}^n \ind{\tilde{V_i} < K} > \theta n \right) \leq P( \tilde{\s}_{\lceil \theta n \rceil } \leq n ). 
\]
Since $\tilde{\s}_k$ is the sum of $k$ i.i.d.\ random variables, Cramer's Theorem implies that this last probability decays on an exponential scale like $e^{-n \theta I_{\tilde{\s}}(1/\theta)}$, where $I_{\tilde{\s}}$ is the large deviation rate function for $\tilde{\s}_k/k$. 

Recalling \eqref{Kdec}, we see that $\L_1(\l) \leq - \min \{ \l K(\theta-1), \, \theta I_{\tilde{\s}}(1/\theta) \}$. Optimizing over $\theta \in (0,1)$ gives
\be\label{supmin}
 \L_1(\l) \leq - \sup_{\theta \in (0,1)} \min\{ \l K(\theta -1), \, \theta I_{\tilde{\s}}(1/\theta) \}. 
\ee
The modified Markov chain $\tilde{V}_i$ inherits the same recurrence/transience properties that $V_i$ has. In particular, $\tilde{V}_i$ is null-recurrent if $\d \in [0,1]$ and transient if $\d < 0$. In either case $E[\tilde{\s}_1] = \infty$ and so it can be shown that $I_{\tilde{\s}}(x)$ is convex, non-increasing, and $I_{\tilde{\s}}(x)>0$ for $x\in [1,\infty)$. 
Therefore, the function $\theta \mapsto \theta I_{\tilde{\s}}(1/\theta)$ is convex and strictly increasing on $(0,1)$ and approaches $0$ as $\theta \ra 0$. Thus, there exists an inverse function $h$ so that $h(x) I_{\tilde{\s}}(1/h(x)) = x$ and $h(x) \ra 0$ as $x \ra 0$. 
We will use this information to analyze the upper bound in \eqref{supmin}.

\begin{figure}
\includegraphics{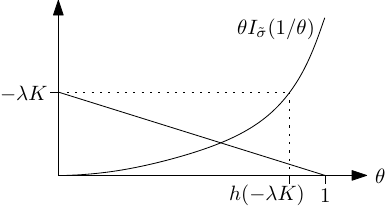}
\caption{For any fixed $\l<0$, the supremum in \eqref{supmin} is attained at the intersection of the two curves. A lower bound for the supremum is obtained by evaluating the line $\l K(\theta -1)$ at $\theta = h(-\l K)$. 
\label{grint}}
\end{figure}
Since the first term in the minimum of \eqref{supmin} is decreasing in $\theta$ and the second term in the minimum is increasing in $\theta$, the supremum is obtained for the value of $\theta$ that makes the two terms in the minimum equal. Thus, the supremum is greater than $\l K(h(-\l K)-1)$ (see Figure \ref{grint}) which in turn implies that $\L_1(\l) \leq \l K(1-h(-\l K))$. Therefore, 
\[
 \liminf_{\l \ra 0^-} \frac{\L_1(\l)}{\l} \geq \lim_{\l \ra 0^-} K( 1- h(-\l K)) = K. 
\]
Since the above argument works for any finite $K$, this implies that $\lim_{\l \ra 0^-} \L_1(\l)/\l = \infty$. 
\end{proof}

\section{Large Deviations for Hitting Times}\label{LDPTsec}

The large deviation principles for $n^{-1} \sum_{i=1}^n V_i$ in Theorems \ref{LDPVn0} and \ref{LDPVn} imply a large deviation principle for the hitting times. 

\begin{thm}\label{LDPTn}
 Let $I_T(t) = I_V((t-1)/2)$. Then, $T_n/n$ satisfies a large deviation principle with convex rate function $I_T(t)$. 
That is, 
\be\label{LDPTnlb}
 \liminf_{n\ra\infty} \frac{1}{n} \log P\left( T_n/n \in G \right) \geq - \inf_{x \in G} I_V(x),
\ee
for all open $G$, and 
\be\label{LDPTnub}
 \limsup_{n\ra\infty} \frac{1}{n} \log P\left( T_n/n \in F \right) \leq - \inf_{x \in F} I_V(x),
\ee
for all closed $F$. 
Moreover, the following qualitative properties are true of the rate function $I_T$. 
\begin{enumerate}
 \item $I_T(t)$ is convex, non-increasing, and continuous on $[1,\infty)$, and there exists a $t_2 > 1$ such that $I_T(t)$ is strictly convex and analytic on $(1,t_2)$. 
 \item $I_T(1) = -\log \Ev[\w_0(1)]$ and $\lim_{t\ra 1^+} I_T'(t) = -\infty$. 
 \item If $\d>2$, then $I_T(t) = 0 \iff t \geq 1/v_0$. Moreover, there exists a $t_1 < 1/v_0$ such that $I_T(t)$ is strictly convex and analytic on  $(t_1, 1/v_0)$ and continuously differentiable on $(t_1,\infty)$.
 \item If $\d \leq 2$, then $I_T(t)>0$ for all $t<\infty$ and $\lim_{t\ra\infty} I_T(t) = 0$. 
\end{enumerate}
\end{thm}
\begin{figure}
\includegraphics[width=6in]{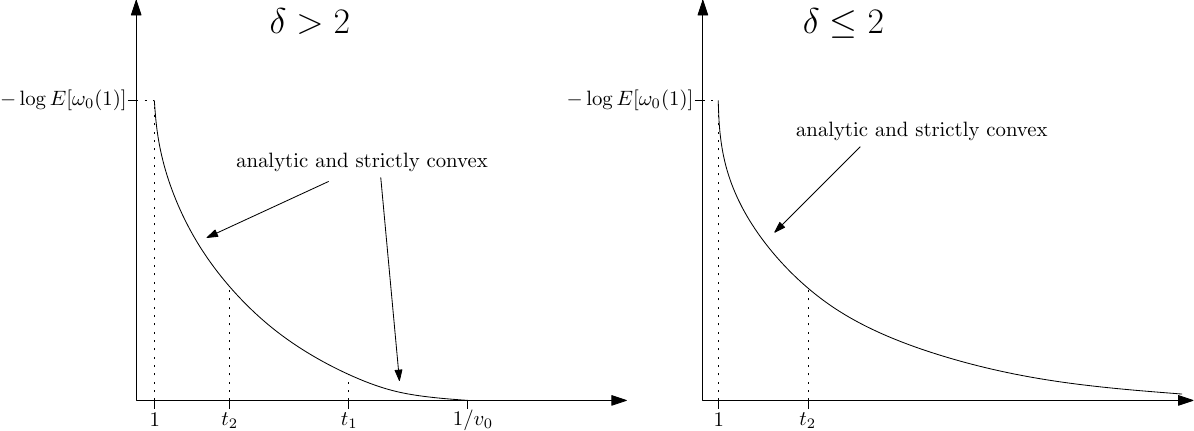}
\caption{A visual depiction of the rate function $I_T$ in the cases $\d>2$ and $\d \leq 2$ showing the qualitative properties stated in Lemma \ref{LDPTn}. 
}
\end{figure}
\begin{proof}
 The properties of the rate function $I_T$ follow directly from the corresponding properties of $I_V$ proved above in Lemmas \ref{infIVlem} and \ref{IVprops}. 
Note that when $\d>2$ we use that the formula for the limiting speed of the excited random walk in \eqref{speedformula} implies that $1/v_0 = E[\s_1 + 2 W_1]/E[\s_1] = 1+2m_0$.

Recall the relationship between the hitting times $T_n$ and the processes $V_i$ and $V_i^{(n)}$ given in \eqref{TnVirep}. 
Then, 
\[
 P( T_n/n \in G ) \geq P\left( 1 + \frac{2}{n} \sum_{i=1}^n V_i \in G, \, V_n = 0 \right), 
\]
since $V^{(n)}_{i} = 0$ for all $i\geq 1$ if $V_0^{(n)}=V_n=0$. The large deviation lower bound \eqref{LDPTnlb} then follows from Theorem \ref{LDPVn0}. 

Since $I_T$ is non-increasing and $\inf_t I_T(t) = 0$, as in the proof of Theorem \ref{LDPVn} the large deviation upper bound will follow from 
\begin{equation}\label{Tndecay}
 \limsup_{n\ra\infty} \frac{1}{n} \log P(T_n \leq nt) = -I_T(t). 
\end{equation}
Again, the relationship between the hitting times $T_n$ and the processes $V_i$ and $V_i^{(n)}$ given in \eqref{TnVirep} implies that
\[
 P(T_n \leq tn) \leq P\left( \sum_{i=1}^n V_i \leq \frac{(t-1)n}{2} \right),
\]
and Theorem \ref{LDPVn} implies that \eqref{Tndecay} holds. 
\end{proof}

To obtain a large deviation principle for the position of the excited random walk we will also need a large deviation principle for the hitting times to the left. However, this is obtained directly as a Corollary of Theorem \ref{LDPTn} by switching the direction of the cookie drifts. 
To be more precise, for any cookie environment $\w = \{\w_i(j)\}$, let $\overline{\w} = \{ \overline{\w}_i(j) \}$ be the associated cookie environment given by $\overline{\w}_i(j) = 1-\w_i(j)$.
Let $\overline{T}_n$ be the hitting times of the excited random walk in the cookie environment $\overline{\w}$. An obvious symmetry coupling gives $T_{-n} = \overline{T}_n$. 
\begin{cor}\label{LDPTnneg}
The random variables $T_{-n}/n$ satisfy a large deviation principle with convex rate function $I_{\overline{T}}$,
where $I_{\overline{T}}$ is the rate function given by Theorem \ref{LDPTn} for the hitting times $\overline{T}_n/n$. 
\end{cor}
\begin{rem}
 Since $\overline{\d} = E[ \sum_{j=1}^M (2\overline{\w}_0(j) - 1) ] = - E[ \sum_{j=1}^M (2\w_0(j) - 1) ] = - \d$, the properties of the rate function $I_{\overline{T}}$ are the same as the properties of the rate function $I_T$ given by Theorem \ref{LDPTn} when $\d $ is replaced by $-\d$. 
For instance, $I_{\overline{T}}(t)>0$ for all $t<\infty$ if $\d \geq -2$. 
\end{rem}

\section{Large deviations for the random walk}\label{LDPXsec}

In this section will show a large deviation principle for $X_n/n$. We begin by defining the rate function $I_X(x)$. 
\be\label{IXdef}
I_X(x) = 
\begin{cases}
 x I_T(1/x) & x > 0 \\
0 & x = 0 \\
|x| I_{\overline{T}}(1/|x|) & x < 0. 
\end{cases}
\ee
Before stating the large deviation principle for $X_n/n$ we will prove some simple facts about the rate function $I_X$. 
\begin{lem}\label{IXprop}
The function $I_X$ is non-negative and continuous on $[-1,1]$ and has the following additional properties
\begin{enumerate}
 \item $I_X(x)$ is non-increasing on $[-1,0]$ and non-decreasing on $[0,1]$. \label{monoprop}
 \item $I_X(x)$ is a convex function. \label{IXconvex}
 \item $I_X(-1) = -\log \Ev[1-\w_0(1)]$ and $I_X(1) = -\log \Ev[\w_0(1)]$. 
 \item There exist $\overline{x}_2 \in (-1,0)$ and $x_2 \in (0,1)$ such that $I_X$ is strictly convex and analytic on $(-1,\overline{x}_2)$ and $(x_2,1)$.  
 \item $\lim_{x \ra -1^+} I'_X(x) = -\infty$ and $\lim_{x \ra 1^-} I'_X(x) = \infty$. 
 \item $I'_X(0) = \lim_{x \ra 0} I_X(x)/x = 0$.
 \item If $\d \in [-2,2]$, then $I_X(x) = 0$ if and only if $x = 0$. 
 \item If $\d > 2$, then $I_X(x) = 0$ if and only if $x \in [0,v_0]$, and there exists an $x_1 \in (v_0, 1)$ such that $I_X$ is strictly convex and analytic on $(v_0,x_1)$ and continuously differentiable on $[0,x_1)$. 
 \item If $\d < -2$ then $I_X(x) = 0$ if and only if $x \in [v_0, 0]$, and there exists an $\overline{x}_1 \in (-1,v_0)$ such that $I_X$ is strictly convex and analytic on $(\overline{x}_1,v_0)$ and continuously differentiable on $(\overline{x}_1,0]$.. 
\end{enumerate}
\end{lem}
\begin{figure}
\includegraphics[width=6in]{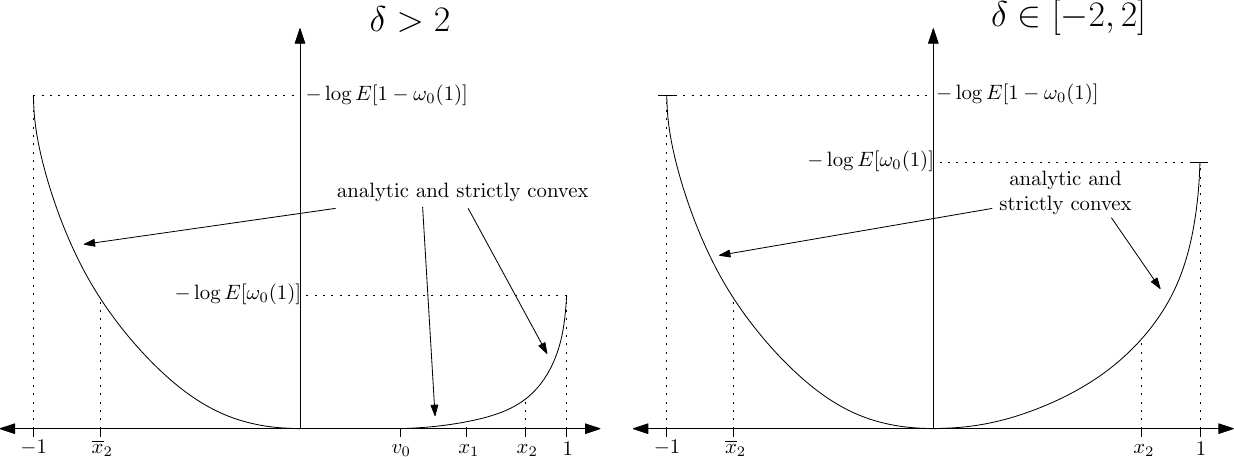}
\caption{A visual depiction of the rate function $I_X$ in the cases $\d>2$ and $\d \in [-2,2]$ showing the qualitative properties stated in Lemma \ref{IXprop}. 
\label{IXfig}}
\end{figure}
\begin{proof}
 Most of the properties in the statement of the Lemma follow directly from the corresponding properties of $I_T$ (or $I_{\overline{T}}$) given by Theorem \ref{LDPTn}, and thus we will content ourselves with only discussing property \ref{IXconvex} from the statement of the Lemma. 

It is a general fact of convex analysis that if $f(x)$ is a convex function on $[1,\infty)$ then $g(x) = x f(1/x)$ is also a convex function on $(0,1]$. Therefore, the convexity of $I_T$ and $I_{\overline{T}}$ imply that $I_X$ is convex on $[-1,0)$ and $(0,1]$, respectively. 
Next, note that $\lim_{x \ra 0^+} I_X(x) = \lim_{x\ra 0^+} x I_T(1/x) = 0$ since $I_T$ is finite and non-increasing, and similarly $\lim_{x\ra 0^-} I_X(x) = 0$. Therefore, $I_X$ is continuous at $x=0$ which in turn implies that $I_X$ is convex on $[-1,0]$ and $[0,1]$, respectively. 
Finally, the convexity of $I_X$ on all of $[-1,1]$ follows from the convexity on $[-1,0]$ and $[0,1]$ and the monotonicity properties in in part \ref{monoprop} of the lemma.
\end{proof}

We now are ready to prove the large deviation principle for the position of the excited random walk. 

\begin{proof}[Proof of Theorem \ref{LDPXn}]
Since the rate function $I_X$ is non-increasing on $[-1,0)$, non-decreasing on $(0,1]$, and $I_X(0) = 0$ it is enough to prove the large deviation upper bound for closed sets of the form $F=[x,1]$ with $x>0$ or $F=[-1,x]$ with $x<0$. 
To this end, let $x>0$ and note that $\{ X_n \geq nx \} \subset \{ T_{\lceil nx \rceil} \leq n \}$. Then, 
Theorem \ref{LDPTn} implies that
\begin{align*}
 \limsup_{n\ra\infty} \frac{1}{n} \log P(X_n \geq nx) &\leq 
\limsup_{n\ra\infty} \frac{1}{n} \log P( T_{\lceil nx \rceil} \leq n ) = -x I_T(1/x)
, \quad \forall x \in (0,1].  
\end{align*}
Similarly, if $x<0$ then $\{X_n \leq nx \} \subset \{ T_{ -\lceil n|x| \rceil} \leq n \}$ and Corollary \ref{LDPTnneg} implies that 
\begin{align*}
 \limsup_{n\ra\infty} \frac{1}{n} \log P(X_n \leq nx) &\leq 
\limsup_{n\ra\infty} \frac{1}{n} \log P( T_{-\lceil n|x| \rceil} \leq n )
= -|x| I_{\overline{T}}(1/|x|)
, \quad \forall x \in [-1,0).
\end{align*}
Recalling the definition of $I_X(x)$ in \eqref{IXdef} and the monotonicity properties of $I_X$ in Lemma \ref{IXprop} finishes the proof of the large deviation upper bound. 

To prove the large deviations lower bound it is enough to show that
\be\label{locallb}
\lim_{\e\ra 0^+} \liminf_{n\ra\infty} \frac{1}{n} \log P( |X_n - nx| < \e n ) \geq -I_X(x), \quad \forall x \in [-1,1],
\ee
First consider the case where $x \in (0,1]$. Then, since the random walk is a nearest neighbor walk
\[
 P( |X_n - nx| < \e n ) \geq P( | T_{\lceil nx \rceil} - n | < \e n - 1).
\] 
Then, Theorem \ref{LDPTn} implies that for any $x \in (0,1]$,
\begin{align*}
 \lim_{\e\ra 0^+} \liminf_{n\ra\infty} \frac{1}{n} \log P( |X_n - nx| < \e n ) 
& \geq  \lim_{\e\ra 0^+} \liminf_{n\ra\infty} \frac{1}{n} \log P( | T_{\lceil nx \rceil} - n | < \e n - 1)\\
&\geq -x I_T(1/x) = -I_X(x), 
\end{align*}
and a similar argument shows that \eqref{locallb} also holds for $x \in [-1,0)$. 
Finally, to show that \eqref{locallb} holds when $x= 0$ note that 
the naive slowdown strategy in Example \ref{slowdownex} implies that 
 $P(|X_n| \leq n^{1/3}) \geq C e^{-c n^{1/3}}$ and thus
\[
 \lim_{\e\ra 0^+} \liminf_{n\ra\infty} \frac{1}{n} \log P( |X_n| < \e n ) = 0 = -I_X(0). 
\]

\end{proof}

\section{Slowdowns}\label{Slowdownsec}

If $\d>2$, then Lemma \ref{IXprop} shows that the rate function $I_X$ is zero in the interval $[0,v_0]$. Thus, probabilities such as $P(X_n < nx)$ decay to zero sub-exponentially for $x \in (0, v_0)$. 
Similarly, since $I_T$ is zero in $[1/v_0,\infty)$ probabilities of the form $P(T_n > nt)$ decay sub-exponentially if $t > 1/v_0$. 
The main goal of this section is to prove Theorem \ref{Slowdownthm} which gives the correct polynomial rate of decay for these probabilities.

In order to prove Theorem \ref{Slowdownthm} we will need the following bound on backtracking probabilities for transient excited random walks. 
\begin{lem}\label{backtracklem}
 Let $\d>1$. Then there exists a constant $C>0$ such that for any $n,r\geq 1$,
\[
 P\left( \inf_{k\geq T_{n+r}} X_k \leq n \right) \leq C r^{1-\d}.
\]
\end{lem}
\begin{rem}
In \cite{bsRGCRW}, Basdevant and Singh showed that such backtracking probabilities could be bounded uniformly in $n$ by a term that vanishes as $r\ra \infty$. However, their argument uses an assumption of non-negativity of the cookie strengths, and their bounds do not give any information on the rate of decay of the probabilities in $r$. 
Our argument is more general (allowing positive and negative cookie drifts) and gives a quantitative rate of decay in $r$. 
\end{rem}

\begin{proof}
First, note that 
\be\label{limTm}
  P\left( \inf_{k\geq T_{n+r}} X_k \leq n \right) = \lim_{m\ra\infty} P\left( \inf_{T_{n+r} \leq k < T_m } X_k \leq n \right)
\ee
The event $\{ \inf_{T_{n+r} \leq k < T_m} X_k \leq n \}$ implies that for every site $i \in [n+1, n+r]$ the excited random walk jumps from $i$ to $i-1$
at least one time before time $T_m$. 
Therefore, 
\begin{align}
P\left( \inf_{T_{n+r} \leq k < T_m } X_k \leq n \right) 
&\leq P( U_i^m \geq 1, \, \forall i \in [n+1,n+r] ) \nonumber \\
&= P( V_i \geq 1, \, \forall i \in [ m-n-r, m-n-1] ). \label{TbackVgap}
\end{align}
Now, the asymptotic age distribution for a discrete renewal process (see Section 6.2 of \cite{lISP}) implies that for any $k\geq 1$
\[
 \lim_{m\ra\infty} P( V_i \neq 0 \, \text{for all } m < i \leq m + k ) = \frac{E[ (\s_1 - k )_+]}{E[\s_1]}. 
\]
Applying this to \eqref{limTm} and \eqref{TbackVgap} we obtain 
\[
 P\left( \inf_{k\geq T_{n+r}} X_k \leq n \right) \leq \frac{E[ (\s_1 - r )_+]}{E[\s_1]}. 
\]
The tail decay of $\s_1$ in \eqref{Wstails} implies that when $\d>1$ there exists a constant $C>0$ such that  $E[(\s_1 - r)_+] \leq C r^{1-\d}$ for any $r\geq 1$. 
\end{proof}

We will also need the following large deviation asymptotics for heavy tailed random variables. 
\begin{lem}\label{heavyLD}
 Let $\{Z_k\}_{k\geq 1}$ be i.i.d.\ non-negative random variables with $P(Z_1 > t) \sim C t^{-\k}$ for some $\k>1$ and $C>0$. Then, 
\[
 \lim_{n\ra\infty} \frac{ \log P\left( \sum_{k=1}^n Z_k > x n \right) }{\log n} = 1-\kappa, \quad \forall x > E[Z_1]. 
\]
\end{lem}
\begin{rem}
 Lemma \ref{heavyLD} is not new, but we provide a quick proof here for the convenience of the reader since we could not find a statement of this lemma in the literature.  
\end{rem}
\begin{proof}
 The statement of the Lemma follows easily from \cite[equation (0.3)]{nLD} when $\k > 2$. Indeed, if $\k>2$ then in fact 
\begin{align*}
P\left( \sum_{k=1}^n Z_k > x n \right) &\sim n P\left(Z_1 - E[Z_1] > n(x-E[Z_1]) \right) \\
&\sim C (x-E[Z_1])^{-\k} n^{1-\k}, \quad \text{as } n \ra\infty,
\end{align*}
for any $x > E[Z_1]$.

When $\k \in (1,2]$ 
we can no longer use \cite[equation (0.3)]{nLD} and so a different approach is needed. 
To this end, first note that since the $Z_k$ are non-negative a simple lower bound is
\[
 P\left( \sum_{k=1}^n Z_k > x n \right) \geq P(\exists k\leq n: \, Z_k > xn) = 1-(1-P(Z_1 > xn))^n. 
\]
Since $1-(1-p)^n \geq np +(np)^2/2$ for any $n\geq 1$ and $p\in [0,1]$ this implies that 
\[
 P\left( \sum_{k=1}^n Z_k > x n \right) \geq n P(Z_1 > xn) + \frac{1}{2}  n^2 P(Z_1 > xn)^2 \sim C x^{-\k} n^{1-\k}. 
\]

To obtain a corresponding upper bound when $\k \in (1,2]$, note that $E[Z_1^\gamma] < \infty$ for any $\gamma \in (0,\k)$. Then, \cite{bROCOM} implies that $P(\sum_{k=1}^n Z_k > x n ) = o(n^{1-\gamma})$ for any $\gamma \in (0,\k)$ and any $x > E[Z_1]$, and this is enough to complete the proof of the lemma. 
\end{proof}

We are now ready to give the proof of the main result of this section. 
\begin{proof}[Proof of Theorem \ref{Slowdownthm}]
We first prove the polynomial rate of decay for the hitting time probabilities in \eqref{Tnslowdown}. 
Since  $\s_k$ and $W_k$ are sums of $k$ i.i.d.\ non-negative random variables with tail decay given by \eqref{Wstails}, Lemma \ref{heavyLD} implies that
\be\label{sNLD}
 \lim_{k\ra\infty} \frac{1}{\log k} P( \s_k > k y) = 1-\d, \quad \text{ if } y > E[\s_1],
\ee
and 
\be\label{WNLD}
\lim_{k\ra\infty} \frac{1}{\log k} P( W_k > k y) = 1-\d/2, \quad \text{ if } y > E[W_1].
\ee

Recall the relationship between the hitting times $T_n$ and the branching processes $V_i$ and $V_i^{(n)}$ given in \eqref{TnVirep}. 
Also, note that the branching process $V_i^{(n)}$ starts with $V_0^{(n)} = V_n$ and has the same offspring distribution as the branching process $V_i$ but without the extra immigrant each generation. Thus, $V_i^{(n)} = 0$ implies that $V_{j}^{(n)}$ for all $j\geq i$ and 
the processes are naturally coupled so that $V_i^{(n)} \leq V_{n+i}$ for all $i\geq 1$. 
Therefore, $T_n$ is stochastically dominated by $n + 2 \sum_{i=1}^{\s_{k(n)}} V_i = n + 2 W_{k(n)}$, where $k(n)$ is defined by $\s_{k(n)-1} < n \leq \s_{k(n)}$. 
Thus, for any $c>0$ 
\begin{align}
P(T_n > nt) &\leq P(k(n) > cn ) + P\left( W_{\fl{cn}} > \frac{n(t-1)}{2} \right)\nonumber \\
 &\leq P( \s_{\fl{cn}} < n ) + P\left( W_{\fl{cn}} > \frac{n(t-1)}{2} \right). \label{rd}
\end{align}
While \eqref{sNLD} implies that the right tail large deviations of $\s_k/k$ decay polynomially, the left tail large deviations decay exponentially since $\s_k$ is the sum of non-negative random variables (use Cramer's theorem). 
That is, 
\[
 \lim_{k\ra\infty} \frac{1}{k} \log P( \s_k < k y ) < 0, \quad \text{ if } y < E[ \s_1 ]. 
\]
Therefore, if we can choose $c$ such that $1/c< E[\s_1]$ and $(t-1)/(2c) > E[W_1]$ the first term in \eqref{rd} will decay exponentially in $n$ while the second term will decay polynomially on the order $n^{1-\d/2}$. 
The assumption that $t>1/v_0 = 1 + 2 E[W_1]/E[\s_1]$ implies that $(t-1)/2 > E[W_1]/E[\s_1]$ and so such a $c$ may be found. 

For a matching lower bound on the polynomial rate of decay of $P(T_n > nt)$, we again use the relationship between the hitting times and the branching process in \eqref{TnVirep} to obtain
\begin{align*}
 P(T_n > nt) 
&\geq P\left( \sum_{i=1}^n V_i > \frac{n(t-1)}{2}\right)\\
&\geq P\left( \exists k\leq n: W_k > \frac{n(t-1)}{2}, \, \s_k \leq n \right) \\
&\geq P\left( W_{c n} > \frac{n(t-1)}{2} \right) - P( \s_{c n} > n ) . 
\end{align*}
If $c < (E[ \s_1 ])^{-1}$ then the assumption that $t > 1/v_0$ implies that $(t-1)/(2c) > E[W_1]$, and so \eqref{WNLD} and \eqref{sNLD} imply that $P(T_n > nt) \geq n^{1-\d/2+o(1)} - n^{1-\d + o(1)} = n^{1-\d/2+o(1)}$. 
This completes the proof of \eqref{Tnslowdown}.

We now turn to the subexponential rate of decay for $P(X_n < xn)$.
A lower bound follows immediately from \eqref{Tnslowdown} since $P(X_n < xn) \geq P( T_{ \lceil xn \rceil } > n )$.
To obtain a corresponding upper bound, note that 
\begin{align}
 P(X_n < xn) 
&\leq P(T_{ \lceil n(x+\e) \rceil} > n ) +  P\left( \inf_{k > T_{\lceil n(x+\e) \rceil}} X_k < xn \right) \nonumber  \\
&\leq P(T_{ \lceil n(x+\e) \rceil} > n ) + C (n \e)^{1-\d}, \label{Xnslowdownub}
\end{align}
where the last inequality follows from Lemma \ref{backtracklem}. 
Now, if $\e>0$ is sufficiently small (so that $x+\e < v_0$) then \eqref{Tnslowdown} implies that the probability in \eqref{Xnslowdownub} is $n^{1-\d/2 + o(1)}$. 
Since $n^{1-\d/2}$ is much larger than $n^{1-\d}$ this completes the proof of the upper bound needed for \eqref{Xnslowdown}. 
\end{proof}

\textbf{Acknowledgements:}
We thank Nina Gantert for pointing out the reference \cite{nnMAP2} to us, Ofer Zeitouni for several helpful discussions regarding large deviations, and Elena Kosygina for helpful discussions regarding the process $V_i$ and the branching process with migration literature.

\bibliographystyle{alpha}
\bibliography{CookieRW}

\end{document}